\documentclass[11pt]{article}
\usepackage{geometry}
\usepackage{graphicx} 
\usepackage{amssymb, amsmath, color, amsfonts, amsthm,verbatim,enumerate,authblk,url}
\usepackage{graphicx, graphics, epsf, epsfig}
\usepackage{longtable,supertabular,multicol}
\usepackage{catchfile}
\usepackage{cite}
\usepackage{xurl}

\newcommand{\Z}{\mathbb Z}
\newcommand{\F}{\mathbb F}
\allowdisplaybreaks

\newtheorem{theorem}{Theorem}

\newtheorem{lemma}[theorem]{Lemma}

\newtheorem{corollary}[theorem]{Corollary}

\newtheorem{conjecture}[theorem]{Conjecture}

\newtheorem{definition}[theorem]{Definition}

\newcounter{encoding}

\newcommand{\inv}{^{-1}}
\newcommand{\srg}{\operatorname{srg}}
\newcommand{\codeg}{\operatorname{codeg}}
\begin{document}

\title{Constructions for Nearly Diagonal Book Ramsey Numbers}
\author[1,2]{Stephen Chung}
\affil[1]{DualVerseAI}
\affil[2]{University of Cambridge}
\author[3]{William J. Wesley}
\affil[3]{Discrete Mathematics Group, Institute for Basic Science (IBS), Daejeon,
South Korea\thanks{This work was supported by the Institute for Basic Science (IBS-R029-C1).}}
\maketitle

\begin{abstract}
    The \emph{Ramsey number} $R(G_1,G_2)$ is the smallest positive integer $n$ such that for every graph $H$ on $n$ vertices, either $H$ contains a copy of $G_1$ or the complement $\overline H$ contains a copy of $G_2$. Ramsey numbers involving the book graphs $B_n$, which consist of $n$ triangles sharing a common edge, have attracted recent attention, especially in the ``nearly diagonal'' case $R(B_{n-1},B_n)$. Rousseau and Sheehan \cite{RousseauSheehanRamseyBooksOriginal} showed the upper bound $R(B_{n-1},B_n) \le 4n-1$ nearly fifty years ago, but no general matching lower bound construction is known. This paper gives several new infinite families of $n$ for which $R(B_{n-1},B_n) = 4n-1$. 

    Some of the main results in this paper were found with assistance from the {\scshape Station}, a multi-agent AI system. These results also appeared in a larger work on autonomous mathematical discovery by the authors and Du \cite{chung2026autonomousmathematicaldiscoveryopenworld}. This paper is a follow-up to that work: we give complete proofs of the book Ramsey results and discuss their connections to objects in combinatorial design theory. 
\end{abstract}

\section{Introduction}

Given two graphs $G_1$ and $G_2$, the \emph{Ramsey number} $R(G_1,G_2)$ is the smallest positive integer $n$ such that for every graph $H$ on $n$ vertices, either $H$ contains a copy of $G_1$ or the complement $\overline H$ contains a copy of $G_2$. Ramsey numbers for various choices of $G_1$ and $G_2$ have been studied extensively both in terms of their exact values and asymptotics, and the dynamic survey of Radziszowski gives the best known bounds \cite{RamseySurvey}. A lower bound $R(G_1,G_2) > n$ can be shown by exhibiting a graph on $n$ vertices that avoids these copies of $G_1$ and $G_2$. We call such a graph a \emph{Ramsey} $(G_1,G_2;n)$ graph. Searches for these constructions form a significant part of the Ramsey number literature, and much computational and theoretical effort has been expended to improve various lower bounds. Random graphs give the best known asymptotic lower bounds for the classical diagonal Ramsey numbers $R(K_s,K_s)$ \cite{SpencerRamseyLB}. But for many choices of $G_1$ and $G_2$, the Ramsey $(G_1,G_2;R(G_1,G_2)-1)$ graphs have interesting structure. Often they come from constructions in other areas of mathematics and have many symmetries and other elegant algebraic or geometric properties.   

We will focus on the family of \emph{book graphs} $B_n = K_2 + \overline{K_n}$, where $+$ denotes the graph join. More simply, $B_n$ is the graph consisting of $n$ triangles that share a common edge or ``spine". The study of Ramsey numbers for book graphs traces back nearly fifty years to Rousseau and Sheehan \cite{RousseauSheehanRamseyBooksOriginal}, and book Ramsey numbers have attracted a lot of attention in the following years for their exact values, asymptotics, and quasirandomness properties \cite{WJWBooks, LidickyMcKinleyPfenderSmallBooksWheels,Conlon_BookRamsey,ConlonFoxWigderson_OffDiagonalBooks,ChenLinRamseyBookUpperBounds,NikiforovRousseauSchelpBookQuasirandomness}. 

The central problem of this work is determining the precise values of $R(B_{n-1},B_n)$. Rousseau and Sheehan \cite{RousseauSheehanRamseyBooksOriginal} proved the upper bound \begin{equation}\label{EqBookUB}
    R(B_{n-1},B_n) \le 4n-1
\end{equation} for all $n$. Constructions of Ramsey $(B_{n-1},B_n;4n-2)$ graphs to give a matching lower bound were given for $n =2,3,4,5$, respectively, by Chv\'atal and Harary in 1972 \cite{ChvatalHararyGenRamsey3}, Clancy in 1977 \cite{ClancySmallRamsey}, Shao, Xu, Bao, and Pan in 2010 \cite{ShaoXuBoPan}, and the second author in 2023 \cite{WJWThesis}.

Two improvements followed shortly when Lidick\'y, McKinley, Pfender, and Van Overberghe \cite{LidickyMcKinleyPfenderSmallBooksWheels} showed computationally that Rousseau and Sheehan's bound is tight for $n \le 21$. The second author independently matched this result and found an infinite family of $n$ for which the bound \eqref{EqBookUB} is tight \cite{WJWBooks}. This evidence led the second author to conjecture the following:

\begin{conjecture}\label{ConjBn-1Bn}
    $R(B_{n-1},B_n) = 4n-1$ for all positive integers $n$. 
\end{conjecture}

This conjecture appeared as a problem on Epoch AI's FrontierMath benchmark \cite{FrontierMathOpenProblems}. Since then, there has been progress from  \cite{TurtureanBookRamsey,ToukoUrsinPersonalComm,PriceBookRamsey, AybarBookRamsey, BhattacharjeeBookRamsey} in finding new infinite families and fixed $n$ values for which the conjecture holds. Other recent work by Dai and Lin \cite{DaiLin2026bookramseynumbersalgebraic} improved lower bounds on $R(B_{n-2},B_n)$ with algebraic constructions. 

\subsection{Main results}

Our results are all new infinite families of $n$ for which $R(B_{n-1},B_n) = 4n-1$. Some of the discoveries in this paper were made with the assistance of the {\scshape{Station}} \cite{chung2026autonomousmathematicaldiscoveryopenworld}, a multi-agent AI system. We discuss the {\scshape Station} briefly in Section \ref{SectionDiscussion}, and refer the reader to \cite{chung2026autonomousmathematicaldiscoveryopenworld} for complete descriptions and analysis of the {\scshape Station}. The theorems here were originally stated, but not proven, in \cite{chung2026autonomousmathematicaldiscoveryopenworld}. The purpose of this article is to give proofs and expound on the mathematical context and details of the problem.  

\begin{theorem} \label{Thm2blocks}
Suppose there exists a conference graph of order $q = 2n-1$. Then $R(B_{n-1},B_n) = 4n-1$. In particular, $R(B_{n-1},B_n) = 4n-1$ if $2n-1$ is a prime power congruent to $1$ modulo $4$.
\end{theorem}

\begin{theorem} \label{Thm4blocks+2}
    Suppose $q = n-1$ is the order of a conference graph. Then $R(B_{n-1},B_n) = 4n-1$. In particular, $R(B_{n-1},B_n) = 4n-1$ if $n-1$ is a prime power congruent to $1$ modulo $4$. 
\end{theorem}

\begin{theorem} \label{ThmDouble}
Suppose $Q > 3$ is a prime power congruent to $3$ modulo $8$, and let $q = \frac{Q-1}{2}$ and $n = q+1$. Then $R(B_{n-1}, B_{n}) = 4n-1$. 
\end{theorem}

\begin{theorem} \label{ThmYP}
Suppose $r \ge 7$ is a prime power congruent to $3$ modulo $4$. Let $m = (r-1)/2$, $q = mr$, and $n = (q+1)/2$. Then $R(B_{n-1}, B_n) = 4n-1$.  
\end{theorem}

\begin{figure} 
\begin{center}
    \includegraphics[width = 0.4\textwidth]{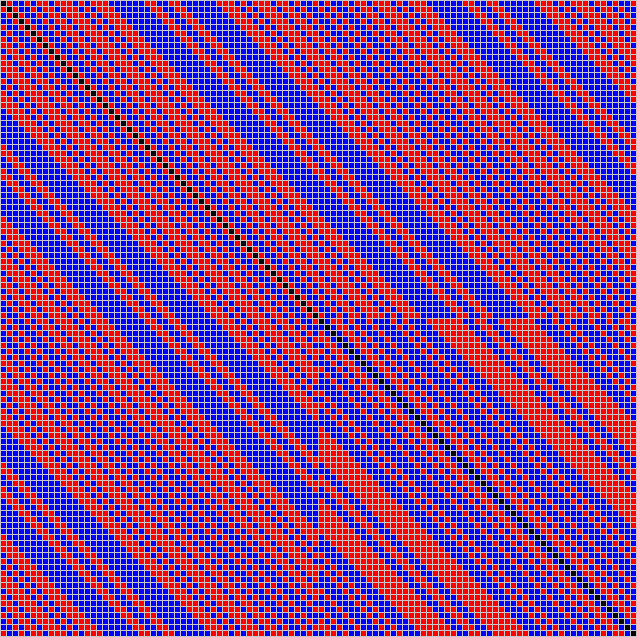}
    \hspace{1em}
    \includegraphics[width = 0.4\textwidth]{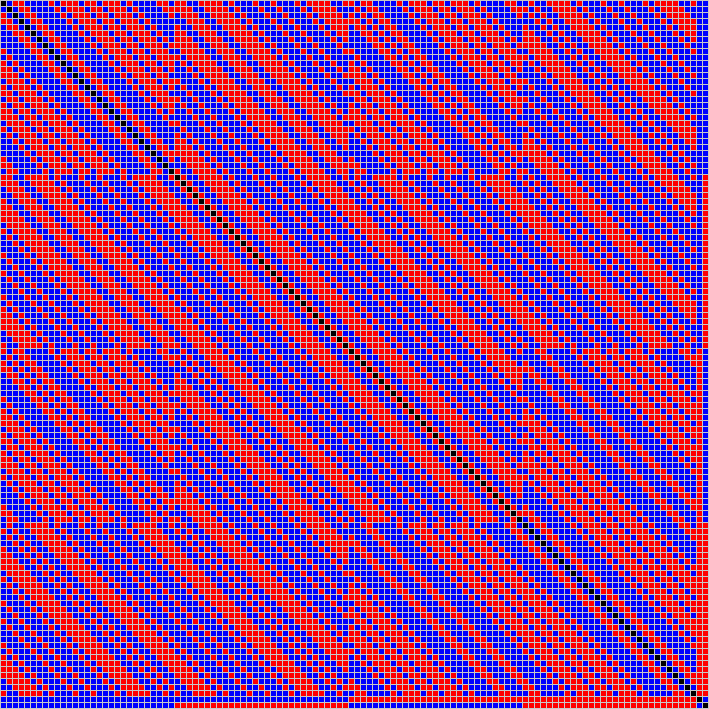}
\end{center}
\begin{center} 
    \includegraphics[width = 0.4\textwidth]{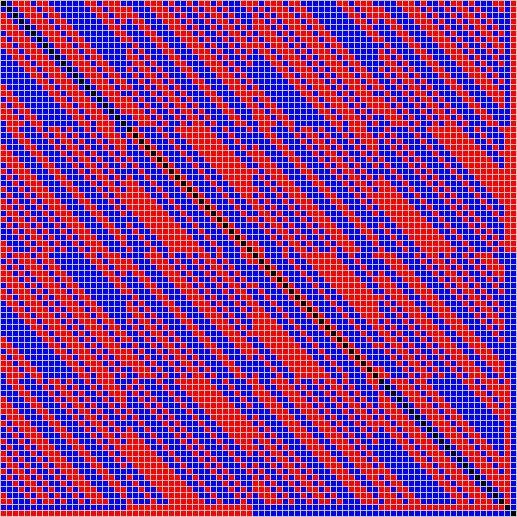}
    \hspace{1em}
    \includegraphics[width = 0.4\textwidth]{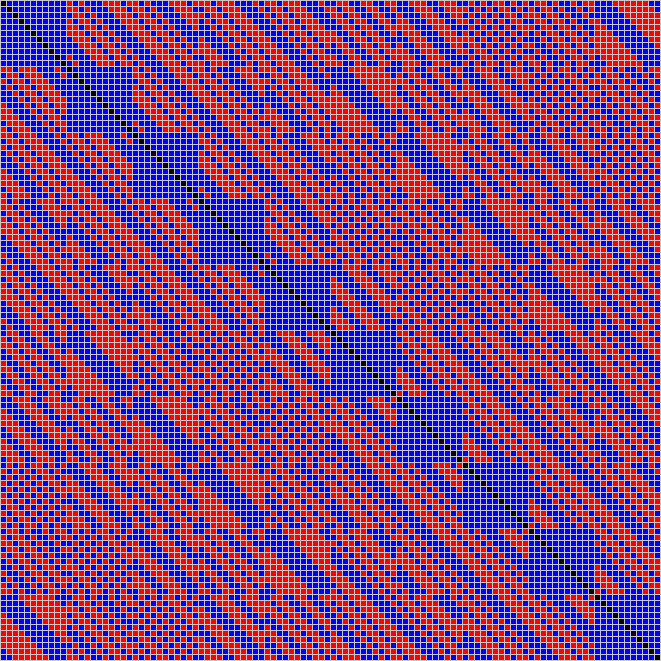}
\end{center}
\caption{From top left to bottom right, adjacency matrices for constructions of Theorem \ref{Thm2blocks} ($n = 27$), Theorem \ref{Thm4blocks+2} ($n = 30)$, Theorem \ref{ThmDouble} ($n = 22)$, and Theorem \ref{ThmYP} ($n = 28$). Entries corresponding to present edges are colored red, and entries corresponding to absent edges are colored blue. Diagonal entries are colored black.}
\label{FigureAdj}
\end{figure}

These results give 31 new exact values of $R(B_{n-1},B_n)$ with $n \le 200$. For each theorem, we give an explicit family of Ramsey $(B_{n-1},B_n;4n-2)$ graphs. Combined with the Rousseau-Sheehan bound \eqref{EqBookUB}, this proves $R(B_{n-1},B_n) = 4n-1$ for each family of $n$. All of the families come from structured graphs that are constructed with objects used in number theory and design theory. This structure is visualized in Figure \ref{FigureAdj}. Theorem \ref{Thm2blocks} is a direct generalization of a result by the second author \cite{WJWBooks} that was originally shown only for Paley graphs. The proof of Theorem \ref{Thm4blocks+2} gives a new ``4-block" lower bound construction involving the Kronecker product of a Hadamard matrix and the Seidel matrix of a conference graph. The constructions for Theorems \ref{ThmDouble} and \ref{ThmYP} come from a particular Legendre pair due to Szekeres \cite{SzekeresTournamentHadamard}. This pair was also used in a different infinite construction for Conjecture \ref{ConjBn-1Bn} found in \cite{AybarBookRamsey} and \cite{TurtureanBookRamsey}. A special case of Theorem 4, $n = 70$, was found independently by Price \cite{PriceBookRamsey}.

We define and give background on these objects in Section \ref{SectionPrelims}. In Section \ref{SectionProofs}, we give the proofs of all main results. We conclude with a discussion on future research directions and the methodology behind finding our results in Section \ref{SectionDiscussion}.

\section{Strongly regular graphs and Legendre pairs}\label{SectionPrelims}
In this section, we introduce strongly regular graphs and Legendre pairs. Both of these are interesting objects of study on their own and have connections to graph and design theory. 

A graph $G$ is \emph{strongly regular} if it is $k$-regular, pairs of adjacent vertices share exactly $\lambda$ common neighbors, and pairs of nonadjacent vertices share exactly $\mu$ common neighbors. We say such a graph on $v$ vertices with these parameters is an $\srg(v,k,\lambda,\mu)$. Strongly regular graphs have remarkable properties and appear in group theory, geometry, and Ramsey theory. A comprehensive monograph on the subject by Brouwer and Van Maldeghem can be found in \cite{BrouwerVanMaldeghemSRGs}. 


Throughout this paper we will let $I = I_n$ be the $n \times n$ identity matrix and $J = J_n$ be the $n \times n$ all-ones matrix. We let $1_n$ denote the $n\times 1$ all-ones matrix. We will omit the subscript when the dimensions of the matrix are clear. It is straightforward to show that the adjacency matrix $A$ of an $\srg(v,k,\lambda,\mu)$ satisfies 
\begin{equation}A^2 = kI + \lambda A + \mu(J-I-A) \label{EqSRGAdjSquared}
\end{equation}using the fact that the number of common neighbors of $i$ and $j$ is $(A^2)_{ij}$. This illustrates the connection between strongly regular graphs and Ramsey graphs for books: an $\srg(v,k,\lambda,\mu)$ contains no $B_{\lambda + 1}$. It also follows  
that if $G$ is an $\srg(v,k,\lambda,\mu)$, then $\overline{G}$ is an $\srg(v,v-k-1,v-2-2k+\mu,v-2k+\lambda$). Therefore a graph with these parameters is a Ramsey $(B_{\lambda+1},B_{v-1-2k+\mu};v)$ graph. Remarkably, several large exact Ramsey numbers for books, such as $R(B_{38},B_{82}) = 244$ and $R(B_{62},B_{65}) = 256$, have strongly regular lower bound witnesses \cite{RousseauSheehanRamseyBooksOriginal,Nikiforov_Rousseau_books} (see also \cite{RamseySurvey}, Table IXb). This phenomenon is not unique to book graphs, and strongly regular graphs often provide good or tight lower bounds for Ramsey numbers involving other graphs. For example, $R(K_{2,n},K_{2,n}) = 4n-2$ if and only if there exists an $\srg(4n-3,2n-2,n-2,n-1)$ \cite{ExooHeikoMengersen_K2n}. Other tight lower bounds for the ``nearly classical" Ramsey numbers $R(K_4-e,K_7-e) = R(K_4-e,K_7) = 28$, and $R(K_5-e,K_7-e) = 65$ come from strongly regular graphs \cite{BlockCircRamseyGoedVanOver, WJW_J4K7}. 

There are many known restrictions on parameters for which strongly regular graphs exist, and many cases are still unresolved \cite{BrouwerSRGparameters}. The \emph{conference graphs} are a special class of strongly regular graphs.
\begin{definition}
    A \emph{conference graph} is an $\srg(v,\frac{v-1}{2},\frac{v-5}{4},\frac{v-1}{4})$ (where $v \equiv 1 \pmod 4$). 
\end{definition}
The \emph{Paley graphs} $P_q$ are perhaps the most famous family of conference graphs with vertex set $\F_q$, where $q \equiv 1 \pmod 4$ is a prime power. Vertices are adjacent if their difference is in $(\F_q^\times)^2$. The Paley graphs $P_5$ and $P_{17}$ give tight lower bounds for the classical Ramsey numbers $R(3,3) = 6$ and $R(4,4) = 18$. However, conference graphs are a richer class of graphs: there exist conference graphs that are not Paley graphs \cite{PeisertSelfComplementarySymmetricGraphs}. Additionally, there are conference graphs of orders for which no Paley graph exists, so Theorem \ref{Thm2blocks} gives a strictly larger family than the one proven in \cite{WJWBooks}. For example, Mathon found a conference graph with 45 vertices in 1978 \cite{MathonConferenceMatrices}, Gritsenko found one with 65 vertices in 2021 \cite{gritsenko2021_conference65}, and Verma recently found one with 85 vertices \cite{verma2026symmetricconferencematrixorder}. If a conference graph on $v$ vertices exists, then a conference graph on $v^r$ vertices exists for all $r \ge 1$ \cite{Turyn_1971}, and there are other infinite families of conference graphs constructed in, for example, \cite{SeberryWhiteman_Conference} and \cite{PolhillConference}. 

Our proofs are most cleanly stated using Seidel matrices. The \emph{Seidel matrix} $S$ of a graph $G$ has zero entries on the diagonal, $S_{ij} = -1$ if $\{i,j\}$ is an edge of $G$, and $S_{ij} = 1$ otherwise. It follows that $S = J-I-2A$, where $A$ is the adjacency matrix of $G$. There is no inherent reason we need to work with $S$ rather than $A$, but this will simplify some calculations.  

The \emph{codegree} $\codeg_G(i,j)$ of a pair of distinct vertices $i,j$ in a graph $G$ is the number of common neighbors of $i$ and $j$. The maximum $r$ such that $B_r$ is a subgraph of $G$ is the maximum codegree among pairs of adjacent vertices in $G$. We begin with a simple lemma that relates the codegrees of pairs of vertices of a graph $G$ to the Seidel matrix of $G$. 
\begin{lemma} \label{LemmaCodegSeidelBound}
    Let $G$ be a graph on $v$ vertices with adjacency matrix $A$ and Seidel matrix $S$, and let $\sigma_r$ be the sum of row $r$ in $S$. If $i$ and $j$ are adjacent in $G$, then 
    \begin{align*}
        \codeg_G(i,j) &= \frac 14 \left(v-4 -\sigma_i-\sigma_j + (S^2)_{ij}  \right).
    \end{align*}
    If $i$ and $j$ are adjacent in $\overline{G}$, then 
    \begin{align*}
        \codeg_{\overline{G}}(i,j) & = \frac 14 \left(v - 4 + \sigma_i+\sigma_j  + (S^2)_{ij}  \right).
    \end{align*} 
\end{lemma}
\begin{proof}
     For the first claim, we have
    \begin{align*}
    \codeg_{G}(i,j) = (A^2)_{ij} &= \sum_{k=1}^v A_{ik}A_{kj} \\ &= \frac 14 \sum_{k \in \{1,\dots,v\}\setminus \{i,j\}} (1-S_{ik})(1-S_{kj}) \\
    &= \frac 14 \left(v-2 -\sigma_i-\sigma_j +2S_{ij} + (S^2)_{ij}  \right)
    \\&= \frac 14 \left(v-4 -\sigma_i-\sigma_j + (S^2)_{ij} \right).
    \end{align*}

    Let $B$ be the adjacency matrix of $\overline{G}$. Similarly, we have 
    \begin{align*}
    \codeg_{\overline{G}}(i,j) = (B^2)_{ij} &= \sum_{k=1}^v B_{ik}B_{kj} \\ &= \frac 14 \sum_{k \in \{1,\dots,v\}\setminus \{i,j\}} (1+S_{ik})(1+S_{kj}) \\
    &= \frac 14 \left(v-2 +\sigma_i+\sigma_j -2S_{ij} + (S^2)_{ij}  \right)\\
    &= \frac 14 \left(v-4 +\sigma_i+\sigma_j + (S^2)_{ij}  \right).
    \end{align*}
\end{proof}

The following corollary is immediate from Lemma \ref{LemmaCodegSeidelBound}, and we will apply it for all of our theorems. 

\begin{corollary}\label{CorMaster}
    If a graph $G$ on $4n-2$ vertices has Seidel matrix $S$ whose row sums are all 1 and satisfies $(S^2)_{ij} \le 0$ for all $i \neq j$, then $G$ is a Ramsey $(B_{n-1},B_n;4n-2)$ graph and $R(B_{n-1},B_n) \ge 4n-1$.
\end{corollary}

The Seidel matrices of conference graphs also have particularly nice properties which follow from \eqref{EqSRGAdjSquared}.

\begin{lemma} \label{LemmaConfSeidelFormula}
    Let $G$ be a conference graph on $v$ vertices, and let $S$ be its Seidel matrix. Then 
    \begin{enumerate}[(i)]
        \item $(S^2)_{ij} = \begin{cases}
        v-1, & \text{ if } i = j, \\
        -1, & \text{ if }i \neq j. 
        \end{cases}$
        \item The sum of the entries in any row of $S$ is 0. 
        \end{enumerate}
\end{lemma}

Another important tool for our constructions is the notion of a \emph{Legendre pair}. A Legendre pair is a pair of functions $x:\Z_\ell \to \{-1,1\}$, $y: \Z_\ell \to \{-1,1\}$ such that

\begin{align*}
    &&\sum_{i=0}^{\ell-1} (x(i)x(i+t)+y(i)y(i+t)) = -2 \text{ for all $t \in \Z_{\ell} \setminus \{0\}.$}
\end{align*}

The term \emph{Legendre pair} was introduced by Fletcher, Gysin, and Seberry in 2001 \cite{FletcherGysinSeberryLegendrePairs}, but an equivalent formulation was studied as early as 1969 by Szekeres \cite{SzekeresTournamentHadamard}. Legendre pairs can be used to construct \emph{Hadamard matrices}, which are square matrices with $\pm1$ entries whose rows and columns are mutually orthogonal. An $n \times n$ Hadamard matrix $H$ thus satisfies $HH^T = nI$. The construction is as follows: Let $x$ and $y$ be a Legendre pair with $\ell$ odd such that $$\sum_{t\in \Z_\ell} x(t) = \sum_{t\in \Z_\ell} y(t) =1.$$  Define $X$ and $Y$ to be circulant matrices given by $X_{ij} = x(j-i)$, $Y_{ij} = y(j-i)$. Then the Legendre pair property gives $XX^T + YY^T = (2\ell+2)I -2J$, and this can be used to show that the $(2\ell+2) \times (2\ell+2)$ matrix 

$$ H = \begin{pmatrix}
    X & Y & 1_\ell & 1_\ell \\
    Y^T & -X^T & 1_\ell & -1_\ell\\
    1_\ell^T & 1_\ell^T &-1 & -1 \\
    1_\ell^T & -1_\ell^T & -1 & 1
\end{pmatrix} $$
is a Hadamard matrix. In particular, we note that $H$ is ``nearly" a Seidel matrix that gives useful bounds in Lemma \ref{LemmaCodegSeidelBound}. Several of our results will use a similar construction pattern, but with some modifications. Constructions using two circulant blocks were also used to construct many of the small $n$ Ramsey graphs in \cite{WJWBooks} and \cite{LidickyMcKinleyPfenderSmallBooksWheels}.

The Legendre pairs in our constructions are built using multiplicative characters of finite fields. Given a finite field $\F_q$ of odd order, the \emph{quadratic character} $\chi : \F_q^\times \to \{-1,1\}$ is given by $\chi(x) = x^{\frac{q-1}{2}}$. We will extend this character to all of $\F_q$ by taking $\chi(0) = 0$. Note that when $q$ is prime, we recover the Legendre symbol: $\chi(a) = \big( \frac{a}{q}\big)$. It is well known that the law of quadratic reciprocity and its supplements generalize to odd prime powers $q$, so that in particular, we have 
   $\chi(-1) = 1$ \text{if and only if $q \equiv 1 \pmod 4$} and 
   $\chi(2) = 1$ \text{if and only if $q \equiv 1,7 \pmod 8$.} 
We will use both these facts in our proofs. Moreover, we require the following results on character sums (see \cite{LidlNiederreiterFiniteFields}, Theorems 5.4 and 5.48 for proofs).

\begin{lemma}\label{LemmaQuadraticCharSum}
    Let $\F_q$ be a finite field with $q$ an odd prime power, and let $\chi$ be the quadratic character of $\F_q$. Then 

    \begin{enumerate}[(i)]
    \item $\sum_{t \in \F_q} \chi(t) = 0.$
    
    \item $\sum_{t\in \F_q} \chi(at^2+bt +c) = -\chi(a)$ if $b^2-4ac \neq 0$. 

    \end{enumerate}
\end{lemma}

\section{Proofs of Main Results}\label{SectionProofs}

Here we give the details of our constructions. For each, we will produce an infinite family of Ramsey $(B_{n-1},B_n;4n-2)$ graphs that satisfy the conditions of Corollary \ref{CorMaster}. The calculations in this section are somewhat tedious and rely on various properties of the objects discussed in Section \ref{SectionPrelims}. The reader who is interested only in verifying correctness can find Lean certificates for the constructions at \url{https://github.com/dualverse-ai/station_data_v2/blob/main/artifacts/book_ramsey/verification.ipynb}. 

\subsection{Theorem \ref{Thm2blocks}}
We now prove Theorem \ref{Thm2blocks}, which is the simplest construction. This generalizes the infinite family from \cite{WJWBooks} and simplifies the proof given there. 

\begin{proof}[Proof of Theorem \ref{Thm2blocks}]
    Let $S$ be the Seidel matrix of a conference graph on $2n-1$ vertices. Let $$H = \begin{pmatrix}
    1 & 1 \\
    1 & -1
\end{pmatrix},\quad C = \begin{pmatrix}
0 & 1 \\
1 & 0
\end{pmatrix}. $$
We claim that $$\Sigma := H \otimes S + C\otimes I = \begin{pmatrix}
    S & S+I\\
    S+I & -S
\end{pmatrix}$$ is the Seidel matrix of a Ramsey $(B_{n-1},B_n;4n-2)$ graph $G$. By Lemma \ref{LemmaConfSeidelFormula}, we have that every row sum of $S$ is 0, so every row sum of $\Sigma$ is 1. Observe that \begin{align*}
    \Sigma^2 &= \begin{pmatrix}
    2S^2 + 2S +I& 0\\
    0 & 2S^2 + 2S+I
    \end{pmatrix}.
\end{align*}
By Lemma \ref{LemmaConfSeidelFormula}, the nondiagonal entries of $S^2$ are equal to -1, and it follows that the nondiagonal entries of $\Sigma^2$ are at most 0. Then by Corollary \ref{CorMaster}, we are done.

\end{proof}

\subsection{Theorem \ref{Thm4blocks+2}}

The proof of Theorem \ref{Thm4blocks+2} is similar to the preceding proof, but involves a slightly more complicated construction. 

\begin{proof}[Proof of Theorem \ref{Thm4blocks+2}]
Let $q = n-1$, and let $S$ be the Seidel matrix of a conference graph on $q$ vertices. Let
    \begin{align*}
    H &= \begin{pmatrix}
        -1 & 1  & -1 & 1 \\
        1  & 1 & 1  & 1 \\
        -1 & 1 & 1 & -1  \\
        1 & 1 & -1 & -1
    \end{pmatrix}, \quad 
    &&C = \begin{pmatrix}
        0 & -1 & -1 & 1 \\
        -1 & 0 & 1 & 1 \\
        -1 & 1 & 0 & 1 \\
        1 & 1 & 1 & 0
    \end{pmatrix}, \\
    U &= \begin{pmatrix}
        1 & 1 \\
        1 & -1 \\
        -1 & 1 \\
        -1 & -1
    \end{pmatrix},
    &&V = U \otimes 1_q, \\
    W &= \begin{pmatrix}
        0 & 1 \\
        1 & 0 
    \end{pmatrix}.
    \end{align*}
    
    We claim that the matrix $$\Sigma :=
    \begin{pmatrix}
        H \otimes S + C \otimes I & V\\
        V^T & W
    \end{pmatrix}$$
    is the Seidel matrix of a Ramsey $(B_{n-1},B_n;4n-2)$ graph $G$.

We first check that the nondiagonal entries of $\Sigma^2$ are nonpositive. A direct calculation applying Lemma \ref{LemmaConfSeidelFormula} shows that 
\begin{align*}
\Sigma^2 
&=\begin{pmatrix}
    4I \otimes S^2 + (HC  + CH)\otimes S + C^2 \otimes I +VV^T & (CU +UW)\otimes 1_q \\
    (CU +UW)^T\otimes 1_q^T& V^TV + W^2 \\
\end{pmatrix}.\\
\end{align*}

We check each of the four blocks: the upper left block is $$\begin{pmatrix}4S^2 + 2S + 3I +2J & 0 & 0 & -2S-2I -2J\\
    0 & 4S^2 + 2S+3I +2J & 2S + 2I -2J& 0 \\
    0 & 2S+2I-2J & 4S^2 + 2S + 3I +2J & 0 \\
    -2S-2I -2J & 0 & 0 &4S^2+2S+3I + 2J
    \end{pmatrix},$$
noting the fact that $H$ is a symmetric Hadamard matrix. All nondiagonal entries of $S^2$ are -1, so all nondiagonal entries in this block are nonpositive.

The upper right block simplifies to $(CU + UW)\otimes 1_q$
using properties of Kronecker products and the fact that $S1_q$ is the zero vector. The latter follows from the definition of conference graphs. We have $$CU+UW = \begin{pmatrix}
    0 & 0 \\
    -4 & 0 \\
    0 & -4 \\
    0 & 0 
\end{pmatrix},$$ so every entry in the upper right block is nonpositive. Since $\Sigma^2$ is symmetric, this follows for the lower left block as well. 

The lower right block simplifies to $(4q+1)I_2$.  Therefore all nondiagonal entries in $\Sigma^2$ are nonpositive. 

Next, we claim that all row sums in $\Sigma$ are 1. This can be seen from the fact that every row sum of $H \otimes S$ is 0 by definition of conference graph and then from adding corresponding rows of $C$ and $U$ and corresponding columns of $U$ and $W$. Corollary \ref{CorMaster} concludes the proof. 

\end{proof}

\subsection{Theorem \ref{ThmDouble}}

The proofs of Theorems \ref{ThmDouble} and \ref{ThmYP} rely on a particular Legendre pair defined in terms of quadratic characters below. 

\begin{definition}\label{Defxy}
Let $Q>3, Q \equiv 3 \pmod 4$ be an odd prime power, and let $\F_Q$ be the finite field of order $Q$. Let $\chi$ be the quadratic character of $\F_Q$, and let $K = (\F_Q^\times)^2$ be the group of nonzero squares in $\F_Q$. 

    Define functions $x,y: K \to \{\pm 1\}$ by
     $x(t) = \chi(1+t)$ for $t \in K$, $y(t) = \chi(1-t)$ for $t \in K \setminus \{1\}$, $y(1) = 1$.
\end{definition}

This pair was first studied by Szekeres \cite{SzekeresTournamentHadamard} to construct difference sets and tournaments with particular properties, and it later appeared in other design theoretic contexts in \cite{ArasuBulutogluHollonLegendreGArray, DingDiffFamilies, YamadaDiffSet}. Lemma \ref{LemmaQuadraticCharSum} is used to prove several properties of $x$ and $y$ which are crucial in the proofs of Theorems \ref{ThmDouble} and \ref{ThmYP}. Parts of the proof of Lemma \ref{Lemmaxyproperties} are essentially the same as results from \cite{SzekeresTournamentHadamard} and \cite{DingDiffFamilies}, but we give the full proof here for self-containment. 

\begin{lemma}\label{Lemmaxyproperties}
     With $x$ and $y$ as in Definition \ref{Defxy}, we have

     \begin{enumerate}[(i)]
         \item $\sum_{t\in K} x(t) = -1$ 
         \item $\sum_{t\in K} y(t) = 1$
         \item $\sum_{t\in K} (x(t)x(\delta t)+y(t)y(\delta t)) = -2$ for all $\delta \in K\setminus \{1\}$
         \item $x(t) = x(t \inv)$ for all $t \in K$
         \item $y(t) = -y(t\inv)$ for all $t \in K\setminus \{1\}$
     \end{enumerate}
        
\end{lemma}

\begin{proof}
    First, observe that the indicator function of $K$ is equal to $
    \frac{1 + \chi(t)}{2}$ for $t \in \F_Q^\times$. Then by Lemma \ref{LemmaQuadraticCharSum}, we have 

    \begin{align*}
        \sum_{t\in K} x(t) &= \frac 12 \sum_{t \in \F_Q^\times} (1+\chi(t))\chi(1+t) \\
        &= \frac 12 \left [ \sum_{t \in \F_Q^\times} \chi(1+t) + \sum_{t\in \F_Q^\times} \chi(t^2+t) \right] \\
        &= \frac 12 ( -\chi(1)  - \chi(1)) \\
        &= -1.
    \end{align*}

    Similarly, we have 

    \begin{align*}
        \sum_{t\in K} y(t) &= y(1) +\frac{1}{2}\sum_{t\in \F_Q^\times \setminus \{1\}}(1+\chi(t))\chi(1-t) \\
        &= 1 + \frac{1}{2}\sum_{t \in \F_Q^\times}\left[\chi(1-t) + \chi(t-t^2) \right] \\ &=
        1 + \frac 1 2(-\chi(1) - \chi(-1)) \\
        &= 1,
    \end{align*}
    noting that $\chi(-1) = -1$ since $Q \equiv 3 \pmod 4$.

    For claim $(iv)$, we have $\chi(t) = 1$, so $x(t\inv) = \chi(1+t\inv) = \chi(t)\chi(1+t\inv) = \chi(t+1)= x(t)$. Similarly, for claim $(v)$ we have $y(t\inv) = \chi(1-t\inv) = \chi(t-1) = \chi(-1)y(t) = -y(t)$.

    For claim $(iii)$, since $K$ is cyclic, we are showing that $x$ and $y$ form a Legendre pair. Using claim $(v)$ and Lemma \ref{LemmaQuadraticCharSum}, we have 

    \begin{align*}
        \sum_{t \in K} [x(t)x(\delta t) + y(t)y(\delta t)] &= y(1)y(\delta) +y(\delta\inv)y(1) +\sum_{t\in K} [\chi((1+t)(1+\delta t)) + \chi((1-t)(1-\delta t))] \\
        &=  \frac 12\sum_{t\in \F_Q^\times} (1+\chi(t))[\chi((1+t)(1+\delta t)) + \chi((1-t)(1-\delta t))] \\
        &=  \frac 12 \sum_{t\in \F_Q^\times}[\chi(1+(1+\delta)t+\delta t^2)+\chi(1-(1+\delta)t + \delta t^2)] \\
        &\qquad +\frac 12 \sum_{t\in \F_Q^\times}[\chi(t(1+t)(1+\delta t)) + \chi(t(1-t)(1-\delta t))] \\
        &= \frac 12 \sum_{t\in \F_Q^\times}[\chi(1+(1+\delta)t+\delta t^2)+\chi(1-(1+\delta)t + \delta t^2)] \\
        &\qquad +\frac 12 \sum_{t\in \F_Q^\times}\chi(t(1+t)(1+\delta t)) + \frac 12 \sum_{u \in \F_Q^\times} \chi(-u(1+u)(1+\delta u)) \\
        &= \frac 12 \sum_{t\in \F_Q^\times}[\chi(1+(1+\delta)t+\delta t^2)+\chi(1-(1+\delta)t + \delta t^2)] \\
        &\qquad +\frac 12 \sum_{t\in \F_Q^\times}\chi(t(1+t)(1+\delta t)) + \frac{\chi(-1)}{2}\sum_{u \in \F_Q^\times} \chi(u(1+u)(1+\delta u)) \\
        &=  \frac 12 \sum_{t\in \F_Q^\times}[\chi(1+(1+\delta)t+\delta t^2)+\chi(1-(1+\delta)t + \delta t^2)] \\
        &= \frac{1}{2}(-2 - 2\chi(\delta)) \\
        &= -2.
    \end{align*}
\end{proof}

We can now prove Theorem \ref{ThmDouble}.

\begin{proof}[Proof of Theorem \ref{ThmDouble}]
Let $g$ be a group generator of $K := (\F_Q^\times)^2 \cong \Z_q$. Define matrices $X$ and $Y$ with $\Z_q$-indexed rows by $X_{ij} = x(g^{j-i}), Y_{ij} = y(g^{j-i})$, where $x$ and $y$ are the functions from Lemma \ref{Lemmaxyproperties}. 

Set $T = \begin{pmatrix}
    X + I & Y \\
    Y^T & -X-I
\end{pmatrix}$, and observe that $T$ is symmetric since $X$ is symmetric by Lemma \ref{Lemmaxyproperties}, part $(iv)$. We have $$T^2 = \begin{pmatrix}
    X^2 +2X + I + YY^T & XY -YX \\
    Y^TX-XY^T & Y^TY +X^2+ 2X + I 
\end{pmatrix} $$

Observe that for $i \neq j$, using Lemma \ref{Lemmaxyproperties} $(iii)$ and $(iv)$, we have \begin{align*}
    (X^2 + YY^T)_{ij} &= \sum_{k \in \Z_q} x(g^{k-i})x(g^{j-k}) + y(g^{k-i})y(g^{k-j})\\
    &= \sum_{k \in \Z_q} x(g^{k-i})x(g^{k-j}) + y(g^{k-i})y(g^{k-j}) \\
    &= -2.
\end{align*}
Similarly, $(Y^TY+X^2)_{ij} = -2$ for $i \neq j$. When $i = j$, we have $(X^2 + YY^T)_{ij} = (Y^TY+X^2)_{ij} = 2q$.

Now observe that \begin{align*}
    (XY - YX)_{ij} &= \sum_{k \in \Z_q} x(g^{k-i})y(g^{j-k}) - \sum_{k\in \Z_q}y(g^{k-i})x(g^{j-k}) = 0,
\end{align*}
by using the fact that $k\mapsto -k +j+i$ is a bijection of $\Z_q$ to reindex the second sum. Similarly, using the fact that $X = X^T$ shows $Y^TX-XY^T = 0$.  

Combined, we have 
\begin{equation}\label{EqT^2}
    T^2 = (2q+3)I_{2q} -2I_2 \otimes(J_q-X).
\end{equation}

We will finally construct our desired $(4q+2) \times (4q+2)$ Seidel matrix $\Sigma$ and apply Corollary \ref{CorMaster}. 

Let \begin{align*}
R &= \begin{pmatrix}
    -I_q & 0 \\
    0 & I_q
\end{pmatrix}, U = \begin{pmatrix}
    T & T + R \\
    T+R & -T
\end{pmatrix},\\ V&= \begin{pmatrix}
    0 & 1 \\
    1 & 0
\end{pmatrix}, L = \begin{pmatrix}
    1 & -1 \\
    -1 & -1 \\
    1 & 1 \\
    -1 & 1 \\
\end{pmatrix} \otimes 1_q,\Sigma = \begin{pmatrix}
    U & L \\
    L^T & V
\end{pmatrix}.
\end{align*}

Since $Q \equiv 3 \pmod 8$, we have $\chi(2) = -1$. Thus $T_{ii} = 0$ for all $i$, and since $T$ is symmetric, one can check that $\Sigma$ is a Seidel matrix. By Lemma \ref{Lemmaxyproperties}, all row sums in $X$ are -1 and all row sums in $Y$ and $Y^T$ are 1. It is then straightforward to verify that all row sums in $\Sigma$ are 1. Now we can proceed as in the proof of Theorem \ref{Thm4blocks+2} and check that the nondiagonal entries of $\Sigma^2$ are nonpositive. 

We have $$\Sigma^2 = \begin{pmatrix}
    U^2 + LL^T & UL+LV \\
    L^T U+VL^T & L^TL+V^2 \\
\end{pmatrix}.$$

Now more calculations give \begin{align*} U^2 &= \begin{pmatrix}
    2T^2 + TR+RT+R^2 & TR-RT \\
    RT-TR & 2T^2 + TR+RT+R^2 
\end{pmatrix}, \hspace{-3em}&TR+RT = -2I_2 \otimes (X+I), \\ TR - RT &= \begin{pmatrix}0 & 2Y\\
-2Y^T & 0 \end{pmatrix}, &LL^T = \begin{pmatrix} 2 & 0 & 0 & -2 \\
0 & 2 & -2 & 0 \\
0 & -2 & 2 & 0 \\
-2 & 0 & 0 & 2\end{pmatrix} \otimes J_q, \\
R^2 &= I_{2q}.
\end{align*}

Write $U^2 + LL^T = \begin{pmatrix}
B_{11} & B_{12} \\
B_{21} & B_{22}
\end{pmatrix}$
with each $B_{ij}$ a $2q \times 2q$ matrix. Then \begin{align*}
    B_{11} = B_{22} &= (4q+5)I_{2q} -I_2 \otimes(2J_q-2X), & B_{12} = \begin{pmatrix}
    0 & 2Y - 2J_q \\ 
    -2Y^T - 2J_q & 0
\end{pmatrix}, \\
B_{21} &= \begin{pmatrix}
    0 & -2Y - 2J_q \\ 
    2Y^T - 2J_q & 0
\end{pmatrix},
\end{align*}
and each nondiagonal entry is nonpositive. 

We can compute the entries of $UL+LV$ using Lemma \ref{Lemmaxyproperties}: 

$$UL + LV = \begin{pmatrix}
    -4 & 0 \\
    0 & 0 \\
    0 & 0 \\
    0 & -4 \\
\end{pmatrix} \otimes 1_q,$$
and each entry is clearly nonpositive. The nonpositivity of the lower left follows from the fact that $\Sigma$ is symmetric. 

Finally, we have $L^T L  + V^2 = (4q+1)I_2$, which is nonpositive away from the diagonal. Applying Corollary \ref{CorMaster} proves the theorem.


\end{proof}

\subsection{Theorem \ref{ThmYP}}

The proof of Theorem \ref{ThmYP} uses similar ideas to that of Theorem \ref{ThmDouble}, but we require some additional notation. 

We will again let $\chi$ be the quadratic character of $\F_r$, $K = (\F_r^\times)^2$, and set $x(t) = \chi(1+t)$ as in Definition \ref{Defxy}. Let $g$ be a generator of $K$, and set $h_t = \chi(1-g^t)$. Recall that $m = (r-1)/2$. 

We set $G := \Z_m \times \F_r$, and let 
\begin{align*}
X &= \{t \in \Z_m : x(g^t) = 1\},\\
D &= \F_r \setminus (\{0\} \cup \{\pm g^t : t \in X\}),\\
H(a) &= \sum_{t \in \Z_m}h_t h_{t+a}, \\
\Delta(a) &= |D \cap g^a D|. 
\end{align*}

We also define three subsets of $G$: 

\begin{align*}
    A &= \{(a,b): a \neq 0 \text{ and  } (b = 0 \text{ or } h_a\chi(b) = 1)\}, \\
    B &= \{(a,b): a \neq 0 \text{ and  } (b = 0 \text{ or } h_a\chi(b) = -1)\}, \\
    C &= \{(a,b): bg^{-a} \in D\}.
\end{align*}

It is convenient to give some properties of all of these sets before the main proof of Theorem \ref{ThmYP}.

\begin{lemma}\label{LemmaABCDsizes}
    The following hold: 
\begin{enumerate}[(i)]
    \item $|A| = |B| = m^2 - 1$, 
    \item $|C| = m^2 + m$, 
    \item $|D| = m+1$,
    \item $A = -A, \ B = -B.$
    \item $H(a) + 2\Delta(a) = m$ for $a \neq 0$.
\end{enumerate}
\end{lemma}
\begin{proof}
    For $(a,b) \in A$, there are $m-1$ choices for $a$, and by Lemma \ref{LemmaQuadraticCharSum}, for each there are exactly $(r+1)/2 = m+1$ choices for $b$, making a total of $m^2-1$. The proof for $B$ is analogous. 

    For $C$, we first find $|X|$. By Lemma \ref{Lemmaxyproperties}, we have that $\sum_{t \in \Z_m } x(g^t) = -1$, so $|X| = (m-1)/2$. Now we have $g^{t_1} \neq -g^{t_2}$ for all $t_1,t_2\in X$. If not, then $-1 = g^{t_1 - t_2} \in K$, but this is impossible since $r \equiv 3 \pmod 4$. Therefore $|D| = r - 2|X| - 1 = m+1$. For a fixed $a$, we have $|g^aD| =  m+1$. There are $m$ choices for $a$, so $|C| = m^2+m$.

    Since $r \equiv 3 \pmod 4$, we have $\chi(-1) = -1$. Claim $(iv)$ follows from this fact, the multiplicativity of $\chi$, and Lemma \ref{Lemmaxyproperties}. 

For the last claim, we have that \begin{align*}H(a) &= \sum_{t \in \Z_m} h_{t}h_{t+a} \\
&= \sum_{t\in \Z_m} \chi(1-g^t)\chi(1-g^{t+a}) \\
&= \sum_{k \in K} y(k)y(g^ak),
\end{align*}
where the last line follows from Lemma \ref{Lemmaxyproperties}, part $(v)$. Let $\gamma_{\epsilon,\eta}$ denote the number of terms in $\sum_{k \in K} x(k)x(g^ak)$ where $x(k) = \epsilon$ and  $x(g^a k) = \eta$. Since $\sum_{k\in K} x(k) = -1$, there are $\frac{m-1}{2}$ values $k$ with $x(k) = 1$ and $\frac{m+1}{2}$ with $x(k) = -1$. Therefore we have 

\begin{align*} \gamma_{1,-1} &= \gamma_{-1,1} =\frac{m-1}{2} - \gamma_{1,1}, \qquad
\gamma_{-1,-1} = \gamma_{1,1} +1.
\end{align*}
Then \begin{equation*}\sum_{k \in K} x(k)x(g^ak) = \gamma_{1,1} + \gamma_{-1,-1} - \gamma_{1,-1} - \gamma_{-1,1} = 4\gamma_{1,1} + 2-m.
\end{equation*}
We have that $D = \{ \pm k : k \in K,\ x(k) = -1\}$. Then $\Delta(a) = 2\gamma_{-1,-1} =2\gamma_{1,1} +2$. By Lemma \ref{Lemmaxyproperties}, we have 

\begin{align*}
    -2 &= \sum_{k \in K} [x(k)x(g^a k) + y(k) y(g^a k)]\\
    &= 2\Delta(a) -2 - m + H(a),
\end{align*}
and rearranging proves the claim.
\end{proof}

We can now prove Theorem \ref{ThmYP}. 
\begin{proof}[Proof of Theorem \ref{ThmYP}]
    Given a subset $U \subseteq G$, we define the matrix $P_U$ by $(P_U)_{ij} = \textbf{1}_U(j-i)$ (where $\textbf{1}_U$ is the indicator function of $U$). Let $M$ be the following matrix: 
    $$ M = \begin{pmatrix} 
        P_A & P_C \\
        P_C^T & P_B
    \end{pmatrix}.$$
    By Lemma \ref{LemmaABCDsizes}, part $(iv)$, we see that $P_A$ and $P_B$ are symmetric. Since $(0,0) \not \in A \cup B$, we have that $M$ is an adjacency matrix of some graph $\Gamma$. Then $\Sigma = J_{2q} - I_{2q} - 2M$ is the Seidel matrix of $\Gamma$, and we claim that $\Gamma$ is the desired Ramsey $(B_{n-1},B_n;4n-2)$ graph. 

    By Lemma \ref{LemmaABCDsizes}, we see that $\Gamma$ is regular with valence $2m^2+m-1 = q-1$. Each row sum of $\Sigma$ is therefore $2q-1-2(q-1) = 1$. 

    We now compute the (nondiagonal) entries of $$M^2 = \begin{pmatrix}
        P_A^2+P_CP_C^T & P_AP_C+P_CP_B \\
        P_C^TP_A+P_BP_C^T & P_C^TP_C+P_B^2
    \end{pmatrix}.$$

Let $d = j-i \in G$. Then we have for $U \in \{A,B\}$ that  
\begin{align*}
    (P_U^2)_{ij} &= \sum_{k \in G} \textbf{1}_U(k-i)\textbf{1}_U(j-k) \\
    &= \sum_{k \in G} \textbf{1}_U(k-i)\textbf{1}_U(k-j) \\
    &= \sum_{k \in G} \textbf{1}_U(k)\textbf{1}_U(k+d),
\end{align*} 
    using Lemma \ref{LemmaABCDsizes} part $(iv)$ in the second line and reindexing by $k \mapsto k+j$ in the third. Similarly, 

\begin{align*}
    (P_CP_C^T)_{ij} &= \sum_{k \in G} \textbf{1}_C(k-i)\textbf{1}_C(k-j) \\
    &= \sum_{k \in G} \textbf{1}_C(k)\textbf{1}_C(k+d).
\end{align*} 
 
These values therefore measure autocorrelations of $A,B$, and $C$. We will compute these values for all possible choices of $d$. 

\textbf{Case 1: } $d = (a,0)$, $a \neq 0$. 

Then $$(P^2_A)_{i,i+d} = \#\{(\alpha,\beta)\in G: (\alpha,\beta) \in A, (\alpha + a,\beta) \in A\}.$$ For $\beta = 0$, there are $m-2$ choices for $\alpha$ since all values in $\Z_m \setminus \{0,-a\}$ work. Otherwise, we need $h_\alpha = h_{\alpha + a} = \chi(\beta)$. Recall that $\sum_{ \alpha \in \Z_m}h_\alpha h_{\alpha + a} = H(a)$. Each term in the sum is $1$ or $-1$ except when $\alpha \in \{0,-a\}$. The number $\gamma$ of choices of $\alpha$ with $h_\alpha = h_{\alpha + a}$ therefore satisfies $\gamma - (m -2 - \gamma) = H(a)$, so $\gamma = \frac{m-2+H(a)}2$. Then for a fixed $\alpha$, there are exactly $m$ choices of $\beta$ with $\chi(\beta) = h_\alpha$. A similar argument holds for $(P^2_B)_{i,i+d}$. Therefore we have

$$(P^2_A)_{i,i+d} = (P^2_B)_{i,i+d} = m-2 + \frac{m(m-2+H(a))}{2} = \frac{m^2-4+mH(a)}{2}.$$

Now for $(P_CP_C^T)_{i,i+d}$, we need to count the number of $(\alpha,\beta)$ with $\beta g^{-\alpha}, \beta g^{-(\alpha + a)} \in D$. For a given $\alpha$, the number of valid $\beta$ is $|g^{\alpha} D \cap g^{\alpha + a}D|$, which after translating is equal to $\Delta(a)$. Therefore $$(P_CP_C^T)_{i,i+d} = m\Delta(a).$$

\textbf{Case 2:} $d = (0,b)$, $b \neq 0$. 

We have $$(P^2_A)_{i,i+d} = \#\{(\alpha,\beta)\in G: (\alpha,\beta) \in A, (\alpha,\beta+b) \in A\}.$$

For $\beta \not \in \{0,-b\}$, we require $h_\alpha = \chi(\beta) = \chi(\beta + b)$. By Lemma \ref{LemmaQuadraticCharSum}, we have $\sum_{\beta \in \F_r} \chi(\beta^2 +b\beta) = -\chi(1) = -1$. Let $\gamma_{\pm} = \#\{\beta \in \F_r : \chi(\beta) = \pm \chi(\beta + b)\}$. Then we have $\gamma_+ - \gamma_- = -1$. Since $b \neq 0$ and $\chi(\beta)\chi(\beta+b) = 0$ for $\beta = 0, -b$, we have $\gamma_+ -(r-2-\gamma_+) = -1,$ and solving gives $$\gamma_+ = m-1.$$ 

It follows from Lemma \ref{Lemmaxyproperties} that $\sum_{\alpha \in \Z_m} h_\alpha = 0$. Thus for any $\beta$, including $0$ and $-b$, there are $(m-1)/2$ choices of $\alpha$. Then $$(P^2_A)_{i,i+d} = \frac{m^2-1}{2}.$$ A similar argument holds for $(P^2_B)_{i,i+d}$. 

Now for $d = (a,b)$ with $b$ nonzero and $a$ not necessarily zero, we have  \begin{align*}(P_CP_C^T)_{i,i+d} &= \#\{(\alpha,\beta) \in G : \beta g^{-\alpha}, (\beta+b)g^{-\alpha -a} \in D\} \\
&= \#\{(u,v) \in D^2: u = \beta g^{-\alpha}, v = (\beta+b)g^{-\alpha -a} \text{ for some $(\alpha, \beta) \in G$} \}.
\end{align*}
Such a pair $(u,v)$ in the latter set gives $b = g^{\alpha} (g^av-u)$. If $b \neq 0$, then we require $u \neq g^a v$. The number of pairs with $u = g^a v$ is precisely $\Delta(a)$. 

We have also that \begin{equation}\label{EqD^2}\frac{b}{g^av-u} = g^\alpha \in K. \end{equation} Now observing that $D = -D$, we have that $(u,v) \in D^2 $ if and only if $(-u,-v) \in D^2$. But since $\chi(-1) = -1$, it follows that exactly one of $(u,v)$ and $(-u,-v)$ makes \eqref{EqD^2} valid. Therefore \begin{equation}\label{EqP_C_b_nonzero}(P_CP_C^T)_{i,i+d} = \frac{|D^2|-\Delta(a)}{2}\end{equation}
In particular, for $d = (0,b)$, we have \begin{equation*}(P_CP_C^T)_{i,i+d} = \frac{m^2+m}{2}.
\end{equation*}

\textbf{Case 3:} $a,b\neq 0$

For $\epsilon,\eta\in \{\pm1\}$, let $$T_\epsilon = \{0\} \cup \{b \in \F_r : \chi(b) = \epsilon\}, N_{\epsilon,\eta}(b) = \#\{\beta \in \F_r : \beta \in T_\epsilon, \beta + b \in T_\eta\}.$$

Then for $b \neq 0$, we have $$N_{\epsilon,\eta} (b)= \frac 1 4 (r+2-\epsilon \eta +(\eta - \epsilon)\chi(b)).$$

We can see this by writing $\textbf{1}_{T_\epsilon}(b) = \frac{1 + \epsilon \chi(b) + \textbf{1}_{b = 0}}2$ so that by Lemma \ref{LemmaQuadraticCharSum} and the fact that $\chi(-1) = -1$, we have

\begin{align*}
    N_{\epsilon,\eta}(b) &= \sum_{\beta \in \F_r} \textbf{1}_{T_\epsilon}(\beta)\textbf{1}_{T_{\eta}}(\beta + b) \\
    &= \frac 1 4\sum_{\beta \in \F_r}(1+\epsilon\chi(\beta) + \textbf{1}_{\beta = 0})(1+\eta \chi(\beta+b)+\textbf{1}_{\beta+b=0}) \\
    &= \frac 14 \sum_{\beta \in \F_r}(1+\textbf{1}_{\beta = 0} + \textbf{1}_{\beta +b = 0} + \epsilon\chi(\beta) + \eta\chi(\beta + b) 
    +\epsilon \eta \chi(\beta^2 + b \beta) \\ &\qquad + \epsilon\chi(\beta)\textbf{1}_{\beta + b = 0} + \eta\chi(\beta+b)\textbf{1}_{\beta = 0} + \textbf{1}_{\beta = 0}\textbf{1}_{\beta + b = 0}) \\
    &= \frac 14 (r+2 + \epsilon\chi(-b) + \eta\chi(b) - \epsilon\eta) \\
    &= \frac 14 (r+2 -\epsilon \eta + (\eta -\epsilon)\chi(b)).
\end{align*}

Now observe that
\begin{align*}(P_{A}^2)_{i,i+d} &= \sum_{\alpha \in \Z_m \setminus \{0,-a\}} N_{h_\alpha, h_{\alpha + a}}(b) \\
&= \frac 1 4\sum_{\alpha \in \Z_m \setminus \{0,-a\}}(r+2-h_{\alpha } h_{\alpha+a} + (h_{\alpha +a}-h_\alpha)\chi(b) ) \\
&= \frac 14 ((m-2)(r+2) - H(a) +(h_{-a}- h_a)\chi(b)) \\
&= \frac 14 ((m-2)(2m+3) - H(a) -2h_a\chi(b)). 
\end{align*}

Similarly, 

\begin{align*}
    (P_{B}^2)_{i,i+d} &= \sum_{\alpha \in \Z_m \setminus \{0,-a\}} N_{-h_\alpha, -h_{\alpha + a}}(b) \\
    &=\frac 14 ((m-2)(2m+3) - H(a) +2h_a\chi(b)).
\end{align*}

From \eqref{EqP_C_b_nonzero}, we have $$(P_CP_C^T)_{i,i+d} = \frac{|D|^2-\Delta(a)}{2} = \frac{(m+1)^2-\Delta(a)}{2}.$$

Now we turn to the off diagonal blocks of $M^2$. Since $M$ is symmetric, we need only compute $P_AP_C+P_CP_B$.

First, we claim that $P_A$ and $P_C$ commute. For any $i,j$ with $d = j-i$, we have:

\begin{align*}
    (P_AP_C)_{ij} &= \sum_{k \in G}\textbf{1}_A(k-i)\textbf{1}_C(j-k) \\
    &= \sum_{k \in G}\textbf{1}_A(k)\textbf{1}_C(d-k) && (k\mapsto k+i)\\
    &= \sum_{k \in G}\textbf{1}_C(k)\textbf{1}_A(d-k)&& (k\mapsto d-k)\\
    &= (P_CP_A)_{ij}.
\end{align*}

Notice that the function $\textbf{1}_A+\textbf{1}_B$ is equal to $\textbf{1}_{\alpha \neq 0} +\textbf{1}_{(\alpha \neq 0) \wedge (\beta = 0)}$. Therefore by Lemma \ref{LemmaABCDsizes}, we have \begin{align*}
    (P_AP_C + P_CP_B)_{ij} &= (P_C(P_A+P_B))_{ij} \\
    &= \sum_{k \in G} \textbf{1}_C(k-i)(\textbf{1}_A(j-k)+\textbf{1}_B(j-k)) \\
    &= \sum_{k \in G} \textbf{1}_C(k-i)(\textbf{1}_A(k-j)+\textbf{1}_B(k-j)) \\
    &= \sum_{k \in G} \textbf{1}_C(k)(\textbf{1}_A(k-d)+\textbf{1}_B(k-d)) \\
    &= \sum_{\alpha \in \Z_m\setminus \{a\}}\left( \textbf{1}_C(\alpha,b)+\sum_{\beta \in \F_r} \textbf{1}_C(\alpha,\beta) \right)
\end{align*}

When $b = 0$, we have $\textbf{1}_C(\alpha,b) = 0$. Otherwise, for a fixed $b \neq 0$, the number of elements $(\alpha,b) \in C$ is $\frac{m+1}{2}$. This is because as $\alpha$ ranges over $\Z_m$, $bg^{-\alpha}$ ranges over $bK$. Since $D = -D$ and $-1 \not \in K$, we have that $|D\cap K| = |D\cap -K| = \frac{m+1}{2}$. Therefore \begin{equation*} \sum_{\alpha \in \Z_m\setminus \{a\}} \textbf{1}_C(\alpha,b) = \frac{m+1}{2} - \textbf{1}_C(a,b). 
\end{equation*}
For a fixed $\alpha$, there are exactly $|D|$ choices of $\beta$ with $(\alpha, \beta) \in C$. Combining everything, we have 

\begin{align} \label{EqNonDiag}
    (P_AP_C + P_CP_B)_{ij} = \begin{cases}
        m^2-1& \text{if $b = 0$,} \\
        m^2-1 + \frac{m+1}{2} -\textbf{1}_C(a,b) & \text{if $b \neq 0$}.
    \end{cases}
\end{align}

A straightforward calculation gives $\Sigma^2 = 4(M^2+M) +I_{2q}-(2q-2)J_{2q}$. The results of this section can be used to calculate the values of $(\Sigma^2)_{i,i+d}$ for all values of $d$. We have the following values for $(M^2+M)_{i,i+d}$, using $\mp$ and $\pm$ for the upper left and lower right blocks:

\begin{table}[h!] \label{TableDValues}
\begin{center}
\begin{tabular}{|c|c|c|}
\hline
$d$ & diagonal blocks & nondiagonal blocks \\
\hline
$(a,0), a \neq 0$ & $\frac{m^2-4+mH(a)}{2} + m\Delta(a) + 1$& $m^2-1$\\
$(0,b), b \neq 0$ &  $\frac{2m^2+m-1}{2}$ & $\frac{2m^2+m-1}{2}$\\
$(a,b), a,b \neq 0$ & $\frac{2m^2-m-6 - H(a) \mp2h_a\chi(b)}{4}+ \frac{(m+1)^2 -\Delta(a)}{2} + \frac{1\pm h_a\chi(b)}{2}$ & $\frac{2m^2+m-1}{2}$\\
\hline 
\end{tabular}
\end{center}
\end{table}

Note that \eqref{EqNonDiag} holds for $d = (0,0)$ on the nondiagonal blocks, so we are done if we can check that each entry is at most $\frac{q-1}{2} = \frac{2m^2+m-1}{2}$. The only nontrivial cases are the $(a,0)$ and $(a,b)$ diagonal cases. 
After applying Lemma \ref{LemmaABCDsizes}, part $(v)$, the $(a,0)$ and $(a,b)$ diagonal block entries in the table simplify to $m^2-1 < \frac{q-1}{2}$ and $\frac{q-1}{2}$, respectively. Corollary \ref{CorMaster} completes the proof.



\end{proof}



\section{Discussion}\label{SectionDiscussion}

The main results of this paper were found with the {\scshape Station}, an environment consisting of multiple AI agents designed to simulate the scientific discovery process. Agents write, share, and review papers in a stored knowledge base. The {\scshape Station} was tested on many other construction problems besides book Ramsey numbers. A more thorough accounting of the other problems and details on the {\scshape Station}'s architecture can be found in \cite{chung2026autonomousmathematicaldiscoveryopenworld}.  

For the book Ramsey problem, two {\scshape Station} instances were used. The first had no internet access and discovered the construction in Theorem \ref{Thm4blocks+2}. We ran the second instance with access to the internet and the results from the first. Agents then discovered, autonomously, the infinite family from Theorem \ref{ThmDouble} and produced several new finite constructions. Some of these latter constructions were special cases of Theorem \ref{ThmYP}, and the full theorem was proven after human review. Theorem \ref{Thm2blocks} was noticed independently by the {\scshape Station} and the human authors. The {\scshape Station} produced proofs for each of the main results, which were then synthesized and rewritten by the authors. 

The constructions in this article highlight the usefulness of design theoretic tools in Ramsey theory. Undoubtedly, there are more connections to the book Ramsey numbers $R(B_m,B_n)$, not only the case $m = n-1$ (see \cite{DaiLin2026bookramseynumbersalgebraic}, for example). More work can be done to find other infinite families with tight lower bounds. It would also be interesting to know if any of these constructions generalize for book Ramsey numbers in other ways, for instance in multicolor versions $R(B_r,B_s,B_t)$ or ``mixed" versions $R(B_n,H)$ for some non-book graph $H$. 

The second author is still optimistic that Conjecture \ref{ConjBn-1Bn} is true, but is less confident that a uniform, explicit construction exists for all $n$. Its status is now similar to that of the Hadamard conjecture\footnote{The Hadamard conjecture states that there is a Hadamard matrix of order $4n$ for all $n$. }, which is still open, but has many constructions.

\renewcommand{\newblock}{}
\bibliography{bibliography}
\bibliographystyle{plain}
\appendix 

\section{Table of Values of $R(B_{n-1},B_n)$}
The table on the following pages gives all known values of $R(B_{n-1},B_n)$ for $n \le 200$. In all cases where the value is known, it is equal to $4n-1$. Bolded rows are new results from this paper (or \cite{chung2026autonomousmathematicaldiscoveryopenworld}). The smallest open case is $n = 64$. Conference graph existence is taken from \cite{BrouwerSRGparameters}. 

The third column lists sources for a lower bound construction of a Ramsey $(B_{n-1},B_n;4n-2)$ graph. Some constructions have overlapping values of $n$, though we do not necessarily list all of them for small $n$. The key for the sources is as follows:
\vspace{10em}
\begin{center}
\begin{tabular}{|c|c|}
\hline 
\bf{Label} & \bf{Reference} \\ 
\hline
A& Aybar's repository \cite{AybarBookRamsey} \\
B& Bhattacharjee \cite{BhattacharjeeBookRamsey} \\
C2& Two block conference graph construction from Theorem \ref{Thm2blocks}\\
C4& Four block conference graph construction from Theorem \ref{Thm4blocks+2} \\
CH & Chv\'atal-Harary \cite{ChvatalHararyGenRamsey3}\\ 
Cl & Clancy \cite{ClancySmallRamsey}\\ 
D& Doubled Legendre family from Theorem \ref{ThmDouble} \\
F& Finite construction from Station \cite{chung2026autonomousmathematicaldiscoveryopenworld} \\
L+& Finite family from Lidick\'y et al. \cite{LidickyMcKinleyPfenderSmallBooksWheels} \\
P& Paley construction from \cite{WJWBooks} (superseded by C2) \\
Pr& Price \cite{PriceBookRamsey}\\
S+ & Shao et al. \cite{ShaoXuBoPan} \\ 
Sz& Szekeres-type family from \cite{TurtureanBookRamsey} and \cite{AybarBookRamsey} \\
T& Turturean's progress report \cite{TurtureanBookRamsey} \\
U& Ursin \cite{ToukoUrsinPersonalComm} \\
W1& Construction from \cite{WJWThesis} \\
W2& Finite family from \cite{WJWBooks} \\
W3& Construction found at \cite{WJWRamseyConstructionsGithub}\\
Y& Construction from Theorem \ref{ThmYP} \\
\hline 
\end{tabular}
\end{center}





\begingroup
\small
\setlength{\columnsep}{1.5em}   
\setlength{\tabcolsep}{3pt}
\setlength{\parindent}{0pt}

\twocolumn                      

\tablehead{%
  \hline
  $n$ & Known & Sources \\
  \hline
}
\tabletail{\hline}
\tablelasttail{\hline}

\newcommand{\tablebody}{

2 & Yes & CH \\
3 & Yes & Cl \\
4 & Yes & S+ \\
5 & Yes & W1 \\
6 & Yes & W2, L+, C4, D \\
7 & Yes & W2, L+, P, Sz, A, C2 \\
8 & Yes & W2, L+ \\
9 & Yes & W2, L+, P, C2 \\
10 & Yes & W2, L+, C4, D \\
11 & Yes & W2, L+, Sz, A, Y, F \\  
12 & Yes & W2, L+ \\
13 & Yes & W2, L+, P, C2 \\
14 & Yes & W2, L+, C4, D \\
15 & Yes & W2, L+, P, Sz, A, C2 \\
16 & Yes & W2, L+ \\
17 & Yes & W2, L+, Sz, A \\
18 & Yes & W2, L+, C4 \\
19 & Yes & W2, L+, P, C2 \\
20 & Yes & W2, L+ \\
21 & Yes & L+, P, Sz, A, C2 \\
22 & Yes & T, A, D \\
23 & Yes & T, A, C2 \\
24 & Yes & T, A \\
25 & Yes & P, C2 \\
26 & Yes & T, A, C4 \\
27 & Yes & P, Sz, A, C2 \\
28 & Yes & T, A, Y, F \\  
29 & Yes & T, A \\
30 & Yes & T, A, C4, D \\
31 & Yes & P, C2 \\
32 & Yes & T, A \\
33 & Yes & Sz, A, C2 \\
34 & Yes & T, A, D \\
35 & Yes & Sz, A \\
36 & Yes & T \\
37 & Yes & P, C2 \\
38 & Yes & T, C4 \\
39 & Yes & T \\
40 & Yes & T \\
41 & Yes & P, Sz, A, C2 \\
42 & Yes & T, C4, D \\
43 & Yes & T, C2 \\
44 & Yes & T \\
45 & Yes & P, Sz, A, C2 \\
46 & Yes & T, C4 \\
47 & Yes & T \\
48 & Yes & T \\
49 & Yes & P, C2 \\
50 & Yes & T, C4 \\
51 & Yes & P, C2 \\
52 & Yes & T \\
53 & Yes & Sz, A \\
54 & Yes & T, C4, D \\
55 & Yes & P, C2 \\
56 & Yes & T \\
57 & Yes & P, Sz, A, C2 \\
58 & Yes & B \\  
\textbf{59} & \textbf{Yes} & \textbf{C2} \\  
60 & Yes & W3 \\
61 & Yes & P, Sz, A, C2 \\
\textbf{62} & \textbf{Yes} & \textbf{C4} \\  
63 & Yes & P, Sz, A, C2 \\
64 & No & -- \\
65 & No & -- \\
\textbf{66} & \textbf{Yes} & \textbf{C4, D} \\
67 & Yes & F, B \\  
68 & No & -- \\
69 & Yes & P, C2 \\
70 & Yes & D, Pr \\
71 & Yes & Sz, A \\
72 & Yes & W3 \\
73 & No & -- \\
\textbf{74} & \textbf{Yes} & \textbf{C4} \\  
75 & Yes & P, C2 \\
76 & No & -- \\
77 & Yes & Sz, A, C2 \\  
78 & No & -- \\
79 & Yes & P, C2 \\
80 & No & -- \\
81 & No & -- \\
\textbf{82} & \textbf{Yes} & \textbf{C4, D} \\
83 & Yes & Sz, A \\
84 & No & -- \\
85 & Yes & P, C2 \\
\textbf{86} & \textbf{Yes} & \textbf{C4, Y, F} \\  
87 & Yes & P, Sz, A, C2 \\
88 & No & -- \\
89 & No & -- \\
\textbf{90} & \textbf{Yes} & \textbf{C4, D} \\
91 & Yes & P, C2 \\
92 & No & -- \\
93 & Yes & U, C2 \\  
94 & No & -- \\
95 & Yes & Sz, A \\
96 & No & -- \\
97 & Yes & P, C2 \\
\textbf{98} & \textbf{Yes} & \textbf{C4} \\  
99 & Yes & P, C2 \\
100 & No & -- \\
101 & No & -- \\
\textbf{102} & \textbf{Yes} & \textbf{C4} \\
103 & No & -- \\
104 & No & -- \\
105 & Yes & Sz, A \\
\textbf{106} & \textbf{Yes} & \textbf{D} \\
107 & No & -- \\
108 & No & -- \\
109 & No & -- \\
\textbf{110} & \textbf{Yes} & \textbf{C4} \\  
111 & Yes & Sz, A \\
112 & No & -- \\
\textbf{113} & \textbf{Yes} & \textbf{C2} \\  
\textbf{114} & \textbf{Yes} & \textbf{C4, D} \\
115 & Yes & P, C2 \\
116 & No & -- \\
117 & Yes & P, Sz, A, C2 \\
\textbf{118} & \textbf{Yes} & \textbf{C4} \\  
119 & No & -- \\
120 & No & -- \\
121 & Yes & P, C2 \\
\textbf{122} & \textbf{Yes} & \textbf{C4, D} \\
123 & Yes & Sz, A \\
124 & No & -- \\
125 & Yes & Sz, A \\
\textbf{126} & \textbf{Yes} & \textbf{C4, D} \\
\textbf{127} & \textbf{Yes} & \textbf{Y, F} \\  
128 & No & -- \\
129 & Yes & P, C2 \\
130 & No & -- \\
131 & Yes & Sz, A \\
132 & No & -- \\
133 & No & -- \\
134 & No & -- \\
135 & Yes & P, C2 \\
136 & No & -- \\
137 & Yes & Sz, A \\
\textbf{138} & \textbf{Yes} & \textbf{C4} \\  
139 & Yes & P, C2 \\
140 & No & -- \\
141 & Yes & P, Sz, A, C2 \\
\textbf{142} & \textbf{Yes} & \textbf{D} \\
143 & Yes & Sz, A \\
144 & No & -- \\
145 & Yes & P, C2 \\
146 & No & -- \\
147 & Yes & P, Sz, A, C2 \\
148 & No & -- \\
149 & No & -- \\
\textbf{150} & \textbf{Yes} & \textbf{C4} \\  
151 & No & -- \\
152 & No & -- \\
153 & Yes & U \\
\textbf{154} & \textbf{Yes} & \textbf{C4, D} \\  
155 & Yes & Sz, A \\
156 & No & -- \\
157 & Yes & P, C2 \\
\textbf{158} & \textbf{Yes} & \textbf{C4} \\  
159 & Yes & P, C2 \\
160 & No & -- \\
161 & Yes & Sz, A \\
162 & No & -- \\
163 & No & -- \\
164 & No & -- \\
165 & Yes & Sz, A \\
\textbf{166} & \textbf{Yes} & \textbf{D} \\
167 & No & -- \\
168 & No & -- \\
169 & Yes & P, C2 \\
\textbf{170} & \textbf{Yes} & \textbf{C4} \\  
171 & Yes & Sz, A \\
172 & No & -- \\
173 & Yes & Sz, A \\
\textbf{174} & \textbf{Yes} & \textbf{C4, D} \\
175 & Yes & P, C2 \\
\textbf{176} & \textbf{Yes} & \textbf{Y} \\
177 & Yes & P, C2 \\
178 & No & -- \\
179 & No & -- \\
180 & No & -- \\
181 & Yes & P, C2 \\
\textbf{182} & \textbf{Yes} & \textbf{C4} \\  
183 & No & -- \\
184 & No & -- \\
185 & Yes & Sz, A \\
\textbf{186} & \textbf{Yes} & \textbf{C4} \\  
187 & Yes & P, C2 \\
188 & No & -- \\
189 & No & -- \\
\textbf{190} & \textbf{Yes} & \textbf{D} \\  
191 & No & -- \\
192 & No & -- \\
193 & No & -- \\
\textbf{194} & \textbf{Yes} & \textbf{C4} \\  
195 & Yes & P, C2 \\
196 & No & -- \\
197 & Yes & Sz, A \\
\textbf{198} & \textbf{Yes} & \textbf{C4} \\  
199 & Yes & P, C2 \\
200 & No & -- \\}

\begin{center}
\begin{supertabular}{|r|c|l|}

2 & Yes & CH \\
3 & Yes & Cl \\
4 & Yes & S+ \\
5 & Yes & W1 \\
6 & Yes & W2, L+, C4, D \\
7 & Yes & W2, L+, P, Sz, A, C2 \\
8 & Yes & W2, L+ \\
9 & Yes & W2, L+, P, C2 \\
10 & Yes & W2, L+, C4, D \\
11 & Yes & W2, L+, Sz, A, Y, F \\  
12 & Yes & W2, L+ \\
13 & Yes & W2, L+, P, C2 \\
14 & Yes & W2, L+, C4, D \\
15 & Yes & W2, L+, P, Sz, A, C2 \\
16 & Yes & W2, L+ \\
17 & Yes & W2, L+, Sz, A \\
18 & Yes & W2, L+, C4 \\
19 & Yes & W2, L+, P, C2 \\
20 & Yes & W2, L+ \\
21 & Yes & L+, P, Sz, A, C2 \\
22 & Yes & T, A, D \\
23 & Yes & T, A, C2 \\
24 & Yes & T, A \\
25 & Yes & P, C2 \\
26 & Yes & T, A, C4 \\
27 & Yes & P, Sz, A, C2 \\
28 & Yes & T, A, Y, F \\  
29 & Yes & T, A \\
30 & Yes & T, A, C4, D \\
31 & Yes & P, C2 \\
32 & Yes & T, A \\
33 & Yes & Sz, A, C2 \\
34 & Yes & T, A, D \\
35 & Yes & Sz, A \\
36 & Yes & T \\
37 & Yes & P, C2 \\
38 & Yes & T, C4 \\
39 & Yes & T \\
40 & Yes & T \\
41 & Yes & P, Sz, A, C2 \\
42 & Yes & T, C4, D \\
43 & Yes & T, C2 \\
44 & Yes & T \\
45 & Yes & P, Sz, A, C2 \\
46 & Yes & T, C4 \\
47 & Yes & T \\
48 & Yes & T \\
49 & Yes & P, C2 \\
50 & Yes & T, C4 \\
51 & Yes & P, C2 \\
52 & Yes & T \\
53 & Yes & Sz, A \\
54 & Yes & T, C4, D \\
55 & Yes & P, C2 \\
56 & Yes & T \\
57 & Yes & P, Sz, A, C2 \\
58 & Yes & B \\  
\textbf{59} & \textbf{Yes} & \textbf{C2} \\  
60 & Yes & W3 \\
61 & Yes & P, Sz, A, C2 \\
\textbf{62} & \textbf{Yes} & \textbf{C4} \\  
63 & Yes & P, Sz, A, C2 \\
64 & No & -- \\
65 & No & -- \\
\textbf{66} & \textbf{Yes} & \textbf{C4, D} \\
67 & Yes & F, B \\  
68 & No & -- \\
69 & Yes & P, C2 \\
70 & Yes & D, Pr \\
71 & Yes & Sz, A \\
72 & Yes & W3 \\
73 & No & -- \\
\textbf{74} & \textbf{Yes} & \textbf{C4} \\  
75 & Yes & P, C2 \\
76 & No & -- \\
77 & Yes & Sz, A, C2 \\  
78 & No & -- \\
79 & Yes & P, C2 \\
80 & No & -- \\
81 & No & -- \\
\textbf{82} & \textbf{Yes} & \textbf{C4, D} \\
83 & Yes & Sz, A \\
84 & No & -- \\
85 & Yes & P, C2 \\
\textbf{86} & \textbf{Yes} & \textbf{C4, Y, F} \\  
87 & Yes & P, Sz, A, C2 \\
88 & No & -- \\
89 & No & -- \\
\textbf{90} & \textbf{Yes} & \textbf{C4, D} \\
91 & Yes & P, C2 \\
92 & No & -- \\
93 & Yes & U, C2 \\  
94 & No & -- \\
95 & Yes & Sz, A \\
96 & No & -- \\
97 & Yes & P, C2 \\
\textbf{98} & \textbf{Yes} & \textbf{C4} \\  
99 & Yes & P, C2 \\
100 & No & -- \\
101 & No & -- \\
\textbf{102} & \textbf{Yes} & \textbf{C4} \\
103 & No & -- \\
104 & No & -- \\
105 & Yes & Sz, A \\
\textbf{106} & \textbf{Yes} & \textbf{D} \\
107 & No & -- \\
108 & No & -- \\
109 & No & -- \\
\textbf{110} & \textbf{Yes} & \textbf{C4} \\  
111 & Yes & Sz, A \\
112 & No & -- \\
\textbf{113} & \textbf{Yes} & \textbf{C2} \\  
\textbf{114} & \textbf{Yes} & \textbf{C4, D} \\
115 & Yes & P, C2 \\
116 & No & -- \\
117 & Yes & P, Sz, A, C2 \\
\textbf{118} & \textbf{Yes} & \textbf{C4} \\  
119 & No & -- \\
120 & No & -- \\
121 & Yes & P, C2 \\
\textbf{122} & \textbf{Yes} & \textbf{C4, D} \\
123 & Yes & Sz, A \\
124 & No & -- \\
125 & Yes & Sz, A \\
\textbf{126} & \textbf{Yes} & \textbf{C4, D} \\
\textbf{127} & \textbf{Yes} & \textbf{Y, F} \\  
128 & No & -- \\
129 & Yes & P, C2 \\
130 & No & -- \\
131 & Yes & Sz, A \\
132 & No & -- \\
133 & No & -- \\
134 & No & -- \\
135 & Yes & P, C2 \\
136 & No & -- \\
137 & Yes & Sz, A \\
\textbf{138} & \textbf{Yes} & \textbf{C4} \\  
139 & Yes & P, C2 \\
140 & No & -- \\
141 & Yes & P, Sz, A, C2 \\
\textbf{142} & \textbf{Yes} & \textbf{D} \\
143 & Yes & Sz, A \\
144 & No & -- \\
145 & Yes & P, C2 \\
146 & No & -- \\
147 & Yes & P, Sz, A, C2 \\
148 & No & -- \\
149 & No & -- \\
\textbf{150} & \textbf{Yes} & \textbf{C4} \\  
151 & No & -- \\
152 & No & -- \\
153 & Yes & U \\
\textbf{154} & \textbf{Yes} & \textbf{C4, D} \\  
155 & Yes & Sz, A \\
156 & No & -- \\
157 & Yes & P, C2 \\
\textbf{158} & \textbf{Yes} & \textbf{C4} \\  
159 & Yes & P, C2 \\
160 & No & -- \\
161 & Yes & Sz, A \\
162 & No & -- \\
163 & No & -- \\
164 & No & -- \\
165 & Yes & Sz, A \\
\textbf{166} & \textbf{Yes} & \textbf{D} \\
167 & No & -- \\
168 & No & -- \\
169 & Yes & P, C2 \\
\textbf{170} & \textbf{Yes} & \textbf{C4} \\  
171 & Yes & Sz, A \\
172 & No & -- \\
173 & Yes & Sz, A \\
\textbf{174} & \textbf{Yes} & \textbf{C4, D} \\
175 & Yes & P, C2 \\
\textbf{176} & \textbf{Yes} & \textbf{Y} \\
177 & Yes & P, C2 \\
178 & No & -- \\
179 & No & -- \\
180 & No & -- \\
181 & Yes & P, C2 \\
\textbf{182} & \textbf{Yes} & \textbf{C4} \\  
183 & No & -- \\
184 & No & -- \\
185 & Yes & Sz, A \\
\textbf{186} & \textbf{Yes} & \textbf{C4} \\  
187 & Yes & P, C2 \\
188 & No & -- \\
189 & No & -- \\
\textbf{190} & \textbf{Yes} & \textbf{D} \\  
191 & No & -- \\
192 & No & -- \\
193 & No & -- \\
\textbf{194} & \textbf{Yes} & \textbf{C4} \\  
195 & Yes & P, C2 \\
196 & No & -- \\
197 & Yes & Sz, A \\
\textbf{198} & \textbf{Yes} & \textbf{C4} \\  
199 & Yes & P, C2 \\
200 & No & -- \\
\end{supertabular}
\end{center} 
\onecolumn                      
\endgroup

\end{document}